\documentclass[a4paper,11pt]{article}
\usepackage{amsmath}
\usepackage{amssymb}
\usepackage{amsthm,xspace}
\usepackage{graphicx,tikz}

\usepackage{mathrsfs}
\usepackage{bbm}
\usepackage{datetime}
\usepackage[numbers]{natbib}

\theoremstyle{definition}
\newtheorem{definition}{Definition}
\newtheorem{remark}[definition]{Remark}
\newtheorem{example}[definition]{Example}

\theoremstyle{plain}
\newtheorem{theorem}[definition]{Theorem}
\newtheorem{lemma}[definition]{Lemma}
\newtheorem{proposition}[definition]{Proposition}

\newcommand{\R}{\mathbb{R}}
\renewcommand{\SS}{\mathbb{S}}
\newcommand{\E}{{\mathbb E}}
\renewcommand{\P}{{\mathbb P}}
\newcommand{\Q}{{\mathbf Q}}
\renewcommand{\H}{{\mathbf H}}
\newcommand{\hmg}{{\mathcal H}}
\newcommand{\eps}{\varepsilon}

\newcommand{\salg}{\mathfrak{K}}
\newcommand{\dsim}{\stackrel{d}{\sim}}
\newcommand{\thf}{\frac{1}{2}}
\newcommand{\Sphere}{\mathbb{S}}
\newcommand{\Z}{\mathbb{Z}}
\newcommand{\gU}{\mathsf{U}}
\newcommand{\fb}{f_{\mathrm{b}}}
\newcommand{\fbo}{f_{\mathrm{b}}^{\mathrm{o}}}
\newcommand{\kappat}{\tilde{\kappa}}
\newcommand{\EE}{\mathbb{E}}
\DeclareMathOperator{\onev}{{\mathbf{1}}}
\newcommand{\imagi}{\boldsymbol{\imath}}
\newcommand{\HH}{\mathbb{H}}
\newcommand{\sH}{\mathcal{H}}
\newcommand{\sE}{\mathcal{E}}
\newcommand{\sI}{\mathcal{I}}
\newcommand{\sT}{\mathcal{T}}
\newcommand{\aslone}{\text{a.s\. and in } L^1 \text{ as } n \to \infty}
\newcommand{\blambda}{\boldsymbol{\lambda}}
\newcommand{\balpha}{\boldsymbol{\alpha}}

\DeclareMathOperator{\dist}{dist}
\DeclareMathOperator{\poisson}{Po}
\DeclareMathOperator{\tpoisson}{\widetilde{Po}}
\DeclareMathOperator{\Exp}{Exp}
\DeclareMathOperator{\gibbs}{Gibbs}
\DeclareMathOperator{\supp}{supp}
\DeclareMathOperator*{\esssup}{ess\,sup}
\newcommand{\eqinlaw}{\text{\raisebox{0pt}[0pt][0pt]{${}\stackrel{\mathscr{D}}{=}{}$}}}
\newcommand{\dst}{\displaystlye}
\newcommand{\tst}{\textstyle}
\newcommand{\sst}{\scriptstyle}
\newcommand{\ssst}{\scriptscriptstyle}
\newcommand{\tsum}{{\tst \sum}}
\newcommand{\dd}{\mathrm{d}}
\newcommand{\bs}{\boldsymbol}
\newcommand{\BB}{\mathbb{B}}
\newcommand{\NN}{\mathbb{N}}
\newcommand{\PP}{\mathbb{P}}
\newcommand{\RR}{\mathbb{R}}
\newcommand{\Rplus}{\mathbb{R}_{+}}
\newcommand{\Zplus}{\mathbb{Z}_{+}}
\newcommand{\mcf}{\mathcal{F}}
\newcommand{\mcx}{\mathcal{X}}
\newcommand{\dtv}{d_{{\mathrm TV}}}
\newcommand{\dzero}{d}
\newcommand{\done}{d_1}
\newcommand{\dtwo}{d_2}
\newcommand{\dtwobar}{\overline{\hspace*{-1.5pt}d\hspace*{1.5pt}}\hspace*{-1.5pt}_2}
\newcommand{\msa}{\mathscr{A}}
\newcommand{\msd}{\mathscr{D}}
\newcommand{\msl}{\mathscr{L}}
\newcommand{\bphi}{\breve{\phi}}
\newcommand{\bara}{\bar{a}}
\newcommand{\barb}{\bar{b}}
\newcommand{\bard}{\bar{d}}
\newcommand{\barlam}{\bar{\lambda}}
\newcommand{\abs}[1]{\lvert #1 \rvert}
\newcommand{\bigabs}[1]{\bigl| #1 \bigr|}
\newcommand{\Bigabs}[1]{\Bigl| #1 \Bigr|}
\newcommand{\biggabs}[1]{\biggl| #1 \biggr|}
\newcommand{\Biggabs}[1]{\Biggl| #1 \Biggr|}
\newcommand{\norm}[1]{\lVert #1 \rVert}
\newcommand{\bignorm}[1]{\bigl\lVert #1 \bigr\rVert}
\newcommand{\Bignorm}[1]{\Bigl\lVert #1 \Bigr\rVert}
\newcommand{\mca}{\mathcal{A}}
\newcommand{\mcb}{\mathcal{B}}
\newcommand{\mcn}{\mathcal{N}}
\newcommand{\mcw}{\mathcal{W}}
\newcommand{\mfn}{\mathfrak{N}}
\newcommand{\tf}{\tilde{f}}
\newcommand{\tilh}{\tilde{h}}
\newcommand{\tu}{\tilde{u}}
\newcommand{\tbalpha}{\widetilde{\balpha}}
\newcommand{\tlambda}{\tilde{\lambda}}
\newcommand{\tnu}{\tilde{\nu}}
\newcommand{\txi}{\tilde{\xi}}
\newcommand{\tXXi}{\widetilde{\XXi}}
\newcommand{\tEta}{\widetilde{\Eta}}
\newcommand{\tmcx}{\widetilde{\mcx}}
\newcommand{\ftv}{\mathcal{F}_{TV}}
\newcommand{\fone}{\mathcal{F}_{1}}
\newcommand{\ftwo}{\mathcal{F}_{2}}
\newcommand{\XXi}{\Xi}
\newcommand{\ssXXi}{\sst {\Xi \hspace*{-0.25em}\rule{0.02em}{1.0ex}\hspace*{0.22em}}}
\newcommand{\Eta}{\mathrm{H}}
\newcommand{\mvert}{\, \vert \,}
\newcommand{\one}{\mathbbm{1}}
\newcommand{\RD}{\RR^D}
\newcommand{\hbit}{\hspace*{1.5pt}}
\newcommand{\nhbit}{\hspace*{-1.5pt}}
\newcommand{\Leb}{\mathrm{Leb}}

\newcommand*{\bi}{\textbf{S}}
\newcommand*{\SSS}{{\upshape{\textbf{(S)}}}\xspace}
\newcommand*{\RS}{{\upshape{\textbf{(RS)}}}\xspace}
\newcommand*{\UB}{{\upshape{\textbf{(UB)}}}\xspace}
\newcommand*{\RC}{{\upshape{\textbf{(RC)}}}\xspace}
\newcommand*{\IR}{{\upshape{\textbf{(IR)}}}\xspace}
\DeclareMathOperator{\pip}{PIP}

\numberwithin{equation}{section}
\numberwithin{definition}{section}

\newlength{\querylen}
\usepackage{fancybox}
\newcommand{\query}[1]{\medskip \noindent
  \shadowbox{\begin{minipage}[t]{\querylen} {#1}
    \end{minipage}}\medskip}

\begin{document}

\title{Continuum percolation for Gibbs point processes}

\author{Kaspar Stucki\footnote{Current adress: Institut f\"ur Mathematische Stochastik, Goldschmidtstrasse 7,
37077 G\"ottingen, Germany}\ \footnote{Email address: kaspar@stucki.org}\\
University of Bern and University of G\"ottingen}

\date{9 August 2013}
\maketitle

\begin{abstract}
We consider percolation properties of the Boolean model generated by a Gibbs point process and balls with deterministic radius. We show that for a large class of Gibbs point processes there exists a critical activity, such that percolation occurs a.s.\  above criticality.

For locally stable Gibbs point processes we show a converse result, i.e.\ they do not percolate a.s.\ at low activity.
\vspace{.3 cm}


\vspace{.3 cm}

\end{abstract}

\section{Introduction}
\label{sec:introduction}

Let $\Xi$ be a point process in $\R^d$, $d\ge 2$, and fix a $R>0$. Consider the random set $Z_R(\Xi)=\bigcup_{x\in \Xi}\BB_R(x)$, where $\BB_R(x)$ denotes the open ball with radius $R$ around $x$. Each connected component of $Z_R(\Xi)$ is called a \emph{cluster}.
We say that $\Xi$ \emph{percolates} (or \emph{$R$-percolates}) if $Z_R(\Xi)$ contains with positive probability an infinite cluster. In the terminology of \cite{mr96} this is a Boolean percolation model driven by $\Xi$ and deterministic radius distribution $R$.

It is well-known that for Poisson processes there exists a critical intensity $\beta_c$ such that a Poisson processes with intensity $\beta>\beta_c$ percolates a.s.\ and if $\beta <\beta_c$ there is a.s.\ no percolation, see e.g.\ \cite{zs85}, or \cite{penrose91} for a more general Poisson percolation model. 
 
For Gibbs point processes the situation is less clear. 
In \cite{PY09} it is shown that for some two-dimensional pairwise interacting models with an attractive tail, percolation occurs if the activity parameter is large enough, see also \citep{Aristoff12} for a similar result concerning the Strauss hard core process in two dimensions.
There is a related work \cite{jansen12}, which states conditions on the intensity instead of the activity. Furthermore in \cite{muermann75,zessin08} it is shown that for finite-range pairwise interacting model there is no percolation at low activity.

The aim of this work is to extend the results of \cite{PY09,Aristoff12} to any dimension $d\ge 2$ and to very general Gibbs point processes. We give a percolation condition on the conditional intensity of a Gibbs process, which is easily understandable and which is satisfied for most Gibbs processes, provided the percolation radius $R$ is large enough. Our main result on percolation, Theorem~\ref{thm:perc} then states, that there exists a critical activity, such that the Gibbs point processes percolate a.s.\ above criticality. The main idea of the proof is the contour method from statistical physics.

Furthermore we state a result on the absence of percolation for locally stable Gibbs processes. Here the idea of the proof is that a locally stable Gibbs process can be dominated by a Poisson process, and then we use the percolation results available for Poisson processes.

The plan of the paper is as follows. In Section~\ref{sec:pre} we introduce the Gibbsian formalism and give the necessary notations. Our main results are stated in Section~\ref{sec:main_results}, and in Section~\ref{sec:pip} they are applied to pairwise interaction processes. Finally, Section~\ref{sec:proofs} contains the proofs of our main results.

\section{Preliminaries}
\label{sec:pre}

Our state space is $\R^d$, for some $d\ge 2$, with the Borel $\sigma$-algebra. Let $\abs{\Lambda}$ denote the Lebesgue measure of a measurable $\Lambda\subset \R^d$, and let $\alpha_d=\abs{\BB_1(0)}$ be the volume of the unit ball.
For two  measurable sets $\Lambda,\Lambda'\subset \R^d$ denote $\mathrm{dist}(\Lambda,\Lambda')=\inf_{x\in \Lambda, \, y\in \Lambda'}\norm{x-y}$, where $\norm{\cdot}$ denotes the Euclidean norm. If $\Lambda=\{x\}$ we write $\mathrm{dist}(x,\Lambda')$ instead of $\mathrm{dist}\big(\{x\},\Lambda'\big)$.
 
Let $(\mfn,\mcn)$ denote the space of locally finite point measures on $\R^d$ equipped with the $\sigma$-algebra generated by the evaluation maps $[\mfn\ni\xi\mapsto \xi(\Lambda) \in \Zplus]$ for bounded Borel sets $
\Lambda\subset \R^d$. A point process is just a $\mfn$-valued random element. We assume the point processes to be \emph{simple}, i.e.\ do not allow multi-points. Thus we can use set notation, e.g.\ $x\in \xi$ means that the point $x$ lies in the support of the measure $\xi$ and $\abs{\xi}=\xi(\R^d)$ denotes the total number of points in $\xi$. For a measurable $\Lambda\subset \R^d$ let $(\mfn_\Lambda,\mcn_\Lambda)$ be the space of locally finite point measures restricted to $\Lambda$ and its canonical $\sigma$-algebra, respectively. Denote $\xi|_\Lambda$ for the restriction of a point configuration $\xi\in \mfn$ onto $\Lambda$.

To define Gibbs processes on $\R^d$ we use the \emph{Dobrushin-Lanford-Ruelle}-approach of local specifications, see e.g.\ \cite{georgii88, nguyenzessin79,ruelle69}. Fix a bounded measurable $\Lambda\subset \R^d$ and a configuration $\omega\in \mfn_{\Lambda^\mathsf{c}}$, called the \emph{boundary condition}, where  $\Lambda^{\mathsf{c}}=\R^d\setminus \Lambda$.
Furthermore fix a measurable function $\Phi\colon \mfn\to \R\cup\{\infty\}$, called the \emph{potential}. For a $\xi \in \mfn_\Lambda$ let
\begin{equation}
\label{eq:u}
u_{\Lambda,\omega}(\xi)=\exp\bigg( -\sum_{\xi'\subset \xi,\, \xi'\neq \emptyset,\, \omega'\subset \omega}\Phi \big(\xi'\cup\omega'\big)\bigg).
\end{equation} 
Define a probability measure on $\mfn$ by
\begin{align}
\label{eq:loc-spez}
\mu_{\Lambda,\omega}(A)=
\frac{1}{c_{\Lambda,\omega}}\bigg(\one\{\omega\in A\} +& \sum_{k=1}^\infty\frac{1}{k!} \int_\Lambda\cdots\int_\Lambda \one\big\{\{x_1,\dots,x_k\}\cup\omega\in A\big\}\times \nonumber \\
 &  u_{\Lambda,\omega}\big(\{x_1,\dots,x_k\}\big)\;dx_1\cdots dx_k\bigg),
\end{align}
where 
\begin{equation}
\label{eq:pf-def}
c_{\Lambda,\omega}=1+\sum_{k=1}^\infty\frac{1}{k!}\int_\Lambda\cdots\int_\Lambda u_{\Lambda,\omega}\big(\{x_1,\dots,x_k\}\big)\; dx_1\cdots dx_k
\end{equation}
is called the \emph{partition function}. Note that the sum in \eqref{eq:u} and the partition function may not exist. However we tacitly assume that the potential $\Phi$ is chosen such that $\mu_{\Lambda,\omega}$ is well-defined for all bounded measurable $\Lambda\subset \R^d$ and for all $\omega \in \mfn_{\Lambda^{\mathsf{c}}}$; we refer the reader to \cite{ruelle69} or \cite{ddg12} for such and related questions. A probability measure $\mu$ on $\mfn$ is called a \emph{Gibbs measure}, if it satisfies the Dobrushin--Lanford--Ruelle equation
\begin{equation}
\label{eq:dlr}
\mu(A)=\int_{\mfn}\mu_{\Lambda,\xi|_{\Lambda^\mathsf{c}}}(A)\; \mu(d\xi),
\end{equation}
for all $A\in \mcn$ and for all bounded measurable $\Lambda\subset \R^d$. Let $\mathcal{G}(\Phi)$ denote the set of all Gibbs measures corresponding to the potential $\Phi$. It may happen that $\mathcal{G}$ contains more than one Gibbs measure; such an event is called a \emph{phase transition}. A point process $\Xi$ is a \emph{Gibbs point process} with potential $\Phi$ if it has a distribution $\mu\in \mathcal{G}(\Phi)$. The measures in \eqref{eq:loc-spez} are the \emph{local specifications} of $\mu$. They are nothing else than conditional probabilities, i.e.\ if $\Xi\sim \mu$, then we have for all bounded measurable $\Lambda\subset \R^d$, for all boundary conditions $\omega\in \mfn_{\Lambda^\mathsf{c}}$ and for all $A\in \mcn$ 
\begin{equation*}
\P(\Xi\in A\mid \Xi|_{\Lambda^\mathsf{c}}=\omega)=\mu_{\Lambda,\omega}(A).
\end{equation*}
For $x\in \R^d$ and $\xi\in \mfn$ define the \emph{conditional intensity} $\lambda$, see \cite{nguyenzessin79}, as
\begin{equation}
\label{eq:ci-def}
\lambda(x\mid \xi)=\exp\Big(-\sum_{\xi'\subset \xi}\Phi\big(\{x\}\cup \xi'\big)\Big).
\end{equation} 
Fix a bounded domain $\Lambda\subset \R^d$ and a boundary condition $\omega\in \mfn_{\Lambda^\mathsf{c}}$. Then we get from \eqref{eq:u}
\begin{equation}
\label{eq:ci-def2}
\lambda(x\mid \xi\cup\omega)=\frac{u_{\Lambda,\omega}\big(\{x\}\cup\xi\big)}{u_{\Lambda,\omega}\big(\xi\big)},
\end{equation}
for all $x \in \Lambda$ and for all $\xi \in \mfn_\Lambda$ such that $x\notin \xi$ and $u_{\Lambda,\omega}(\xi)>0$. Let $\mu\in \mathcal{G}(\Phi)$ and define $N_\Lambda=\{\xi\in \mfn \colon u_{\Lambda, \xi|_{\Lambda^\mathsf{c}}}(\xi|_\Lambda)=0\}$. Obviously $\mu_{\Lambda,\omega}(N_\Lambda)=0$ for all $\omega\in \mfn_{\Lambda^{\mathsf{c}}}$, and by \eqref{eq:dlr} we get $\mu(N_\Lambda)=0$ as well. Thus \eqref{eq:ci-def2} holds for $\mu$-a.e.\ $\xi\in \mfn$ with $x\notin \xi$.

Note that if we are interested in Gibbs point processes on a bounded domain with empty boundary condition ($\omega=\emptyset$), Equation~\eqref{eq:ci-def2} coincides, up to the null-set $N_\Lambda$, with the definition of the conditional intensity commonly used in spatial statistics, see e.g.\ \cite[Def.~6.1]{moellerwaage04}. 
Roughly speaking, the conditional intensity is the infinitesimal probability that $\Xi$ has a point at $x$, given that $\Xi$ coincides with the configuration $\xi$ everywhere else. 

For the rest of this paper we assume the potential $\Phi$ to be constant for one-point configurations, i.e.\ $\Phi\big(\{x\}\big)=\Phi\big(\{0\}\big)$ for all $x \in \R^d$. Let $\beta=\exp\big(-\Phi\big(\{0\}\big)\big)$ and for $x\in \R^d$ and $\xi \in \mfn$ denote
\begin{equation*}
\tilde{\lambda}(x\mid \xi)=\frac{\lambda(x\mid \xi)}{\beta}=
\exp\Big(-\sum_{\xi'\subset \xi,\, \xi'\neq\emptyset}\Phi\big(\{x\}\cup \xi'\big)\Big).
\end{equation*}
The constant $\beta$ is called \emph{activity parameter} or simply \emph{activity}. To be able to keep track of $\beta$ in our main results, we will mainly use the notation $\beta \tilde{\lambda}$ for the conditional intensity.

In order to describe Gibbs processes by the DLR-approach, one can equivalently characterize them through the conditional intensity, see \cite{nguyenzessin79}. Therefore, denote $\mathcal{G}(\beta,\tilde{\lambda})$ as the set of Gibbs measures corresponding to the conditional intensity $\beta \tilde{\lambda}$.

A Gibbs point process $\Xi\sim \mu\in \mathcal{G}(\Phi)$ is called a \emph{pairwise interaction process} if for every configuration $\xi\in \mfn$ with $\abs{\xi}\ge 3$ we have $\Phi(\xi)=0$. By denoting $\varphi(x,y)=\exp\big(-\Phi\big(\{x,y\}\big)\big)$ the conditional intensity simplifies to
\begin{equation*}
\lambda(x\mid \xi)=\beta\prod_{y\in \xi}\varphi(x,y),
\end{equation*}
for all $x\in \R^d$ and for all $\xi \in \mfn$, with $x\notin \xi$.
With a slight abuse of notation, we shall use $\mathcal{G}(\beta,\varphi)$ for the set of the corresponding Gibbs measures.

One of the most important point processes is surely the \emph{Poisson process}. Here, we concentrate only on homogeneous Poisson processes, which can be defined as pairwise interaction processes with $\varphi\equiv 1$. For Poisson processes the parameter $\beta$ is called \emph{intensity}. The more common definition, however, is the following. A point process $\Pi$ is called a Poisson process with intensity $\beta$ if for bounded and pairwise disjoint sets $\Lambda_1,\dots,\Lambda_n$ the random variables $\Pi(\Lambda_1),\dots,\Pi(\Lambda_n)$ are independent and Poisson distributed with mean $\beta \abs{\Lambda_i}$, for $i=1,\dots,n$. It is an easy exercise to show the equivalence of the two definitions using \eqref{eq:loc-spez}.

For Poisson processes we have the following percolation result, see \cite{zs85}, or also \cite[Sec.~12.10]{grimmett99}.

\begin{proposition}
\label{prop:poisson}
Let $\Pi$ be a Poisson process with intensity $\beta$. For all $R>0$ there exists a critical intensity $0<\beta_c<\infty$ such that $Z_R(\Pi)$ contains a.s.\ only finite clusters if $\beta <\beta_c$ and $Z_R(\Pi)$ has an infinite cluster a.s.\ if $\beta>\beta_c$.
\end{proposition}

\section{Main results}
\label{sec:main_results}

For our main result on percolation we need the following definition. Let $\Xi\sim\mu\in \mathcal{G}(\beta,\tilde{\lambda})$. We say $\Xi$ satisfies \emph{condition {\bf(P)}} with  constants $r,\delta>0$ if 
\begin{equation*}
\tilde{\lambda}(x\mid\xi)\ge \delta \quad \text{for all } x\in \R^d \text{ and for $\mu$-a.e.\ } \xi \in \mfn \text{ with } \mathrm{dist}(x,\xi)\ge r.
\end{equation*}
If for all $\beta>0$, all Gibbs processes with distribution in $\mathcal{G}(\beta,\tilde{\lambda})$ satisfy condition \textbf{(P)}, then we say that $\tilde{\lambda}$ satisfies condition {\bf(P)} itself.


A physical interpretation of condition \textbf{(P)} could be the following. Let $\xi\in \mfn$ and choose a $x\in \R^d$ such that its nearest point in $\xi$ is at a distance at least $r$. Assume that the point process $\Xi$ coincides with $\xi$ everywhere except at $x$. Then the condition \textbf{(P)} states that the \emph{energy cost} of adding a point at $x$ is bounded from above by an universal constant which does not depend on the location $x$ nor on the configuration $\xi$.  

The next theorem is our main result on percolation.

\begin{theorem}
\label{thm:perc}
Let $\tilde{\lambda}$ satisfy condition \textbf{(P)} with constants $r$ and $\delta$. Then for all $R>r$ there exist a $\beta_+<\infty$ such that for all $\beta > \beta_+$ and for all Gibbs processes $\Xi$ with distribution in $\mathcal{G}(\beta,\tilde{\lambda})$, the set $Z_R(\Xi)$ contains a.s.\ an infinite cluster. 
\end{theorem}

\begin{remark}
\label{rem:subspace}
In fact, the proof of Theorem~\ref{thm:perc} yields a slightly  stronger percolation result. Let $\Xi$ be as in Theorem~\ref{thm:perc}, and let $\mathbb{H}\subset \R^d$ be any affine subspace of dimension at least two. Then $Z_R(\Xi)\cap \mathbb{H}$ contains a.s.\ an infinite cluster.
\end{remark}

For our result on non-percolation we need the following stability assumption. Let $\Xi$ be a Gibbs process with distribution $\mu$ and conditional intensity $\lambda$. Then $\Xi$ is called \emph{locally stable} if there exists a constant $c^*$ such that $\lambda(x\mid \xi)\le c^*$ for all $x\in \R^d$ and for $\mu$-a.e.\ $\xi \in \mfn$. If all Gibbs point processes corresponding to the conditional intensity $\lambda$ are locally stable, we call $\lambda$ itself locally stable. Most Gibbs point processes considered in spatial statistics are locally stable, see \cite[p.~84]{moellerwaage04} and \cite[p.~850]{km00}. However the most important example in statistical physics, the \emph{Lennard--Jones process}, is not locally stable.

\begin{theorem}
\label{thm:non-perc}
Let $\beta \tilde{\lambda}$ be locally stable. Then for all $R>0$ there exists a $\beta_->0$ such that for all $\beta < \beta_-$ and for all Gibbs processes $\Xi$ with distribution in $\mathcal{G}(\beta,\tilde{\lambda})$, the set $Z_R(\Xi)$ contains a.s.\ only finite clusters. 
\end{theorem}

\begin{remark}
The proof of Theorem~\ref{thm:non-perc} relies on the fact that a locally stable Gibbs process can be dominated by a Poisson process, and then Proposition~\ref{prop:poisson} is applied. However for Poisson processes there are more general percolation results available, e.g.\ the balls could be replaced by some random geometric objects. Thus Theorem~\ref{thm:non-perc} may be generalized along the lines of \cite{penrose91}.
\end{remark}

The next example combines Theorem~\ref{thm:perc} and Theorem~\ref{thm:non-perc} to get a similar result as in Proposition~\ref{prop:poisson}.

\begin{example}
\label{ex:aip}
A Gibbs point process is called an \emph{area interaction process}, see \cite{bv95}, if its conditional intensity is given by
\begin{equation*}
\lambda(x\mid \xi)=\beta \gamma^{-\abs{\BB_{r_0}(x)\setminus \cup_{y\in \xi}\BB_{r_0}(y)}}
\end{equation*}
for some $\gamma,r_0>0$. One easily gets the estimates
\begin{align*}
1\le \tilde{\lambda}(x\mid \xi) \le \gamma^{-\alpha_d{r_0}^d} \quad &\text{if} \quad 0\le \gamma \le 1,\\
\gamma^{-\alpha_d{r_0}^d}\le \tilde{\lambda}(x\mid \xi) \le 1 \quad \quad \quad \; &\text{if} \quad  \gamma \ge 1.
\end{align*}
Thus $\lambda$ satisfies condition \textbf{(P)} for all $r>0$ with $\delta=\min\{1,\gamma^{-\alpha_d{r_0}^d}\}$ and it is locally stable with constant $c^*=\beta\max\{1,\gamma^{-\alpha_d{r_0}^d}\}$. Then the Theorems~\ref{thm:perc} and \ref{thm:non-perc} yield that for all $R>0$ there exists two constants $0<\beta_-\le \beta_+<\infty$ such that all area interaction processes with $\beta < \beta_-$ do not $R$-percolate a.s.\ and all area interaction processes with $\beta > \beta_+$ do $R$-percolate a.s. It remains open whether $\beta_-=\beta_+$, as for the Poisson process, see Proposition~\ref{prop:poisson}.
\end{example}

\section{Pairwise interaction processes}
\label{sec:pip}

In this section we consider pairwise interaction processes and compare the results of \cite{PY09} with ours. The next proposition shows which pairwise interaction processes  satisfy condition \textbf{(P)}.

\begin{proposition}
\label{prop:pip}
Let $r>0$. Assume that the interaction function $\varphi$ satisfies one of the following conditions.
\begin{itemize}
\item[(i)] For all $x,y\in \R^d$ with $\norm{x-y}\ge r$ we have $\varphi(x,y)\ge 1$.
\item[(ii)] There exist constants  $\tilde{\delta}>0$ and  $r_{\max}<\infty$ such that
\begin{align*}
\varphi(x,y)=0 \quad &\text{if} \quad 0\le \norm{x-y}<r,\\
\varphi(x,y)\ge \tilde{\delta} \quad &\text{if} \quad r\le \norm{x-y} < r_{\max},\\
\varphi(x,y)\ge 1 \quad &\text{if} \quad r_{\max} \le \norm{x-y}.
\end{align*}
\end{itemize}
Then there exists a $\delta>0$ such that all Gibbs processes with distribution in $\mathcal{G}(\beta,\varphi)$ satisfy  condition \textbf{(P)} with constants $r$ and $\delta$.
\end{proposition}

\begin{proof} 
\begin{itemize}
\item[(i)] Since $\tilde{\lambda}(x\mid \xi)=\prod_{y\in \xi} \varphi(x,y)$, condition \textbf{(P)} is trivially satisfied with $\delta=1$.
\item[(ii)]
Assume $\tilde{\delta}<1$, otherwise condition $(i)$ is satisfied. Let $\mu\in \mathcal{G}(\beta,\varphi)$. Fix a $x\in \R^d$ and consider the event
\begin{equation*}
N_{x}=\big\{\xi\in \mfn\colon \text{ There exist }  y,y'\in \xi|_{\BB_{r_{\max}}(x)}   \text{ with }  \norm{y-y'} <r\big\}.
\end{equation*}
For a bounded $\Lambda\supset \BB_{r_{\max}}(x)$ and any boundary condition $\omega\in \mfn_{\Lambda^\mathsf{c}}$ we get $\mu_{\Lambda,\omega}(N_x)=0$ and by \eqref{eq:dlr}  also $\mu(N_x)=0$.
Furthermore there exists a constant $m<\infty$ such that for all $\xi\in \mfn\setminus N_x$ we have $\xi\big(\BB_{r_{\max}}(x)\big)\le m$, and $m$ does not depend on $x$ (e.g.\ $m$ can be  chosen as the maximal number of balls with radius $r/2$ which can be placed in a ball with radius $r_{\max}+r/2$). Thus condition \textbf{(P)} is satisfied with $\delta=\tilde{\delta}^m$ for all $x\in \R^d$ and for $\mu$-a.e.\ $\xi \in \mfn$. 
\end{itemize}
\end{proof}

In \cite{PY09} there are various assumptions on $\varphi$ including our condition $(i)$ of Proposition~\ref{prop:pip} and an attraction condition. Namely, $\varphi$ is said to have an attractive tail if there exists two constants $r_a<r_a'$ such that $\varphi(x,y)>1$, whenever $r_a\le \norm{x-y} \le r_a'$.

Furthermore the authors percolation radius $R$ has to be greater than $\sqrt{2}r$. We could reduce this bound by a factor $\sqrt{2}$, but in some cases, e.g.\ for a \emph{Strauss hard core} process ($\varphi(x,y)=\one\{\norm{x-y}\ge r\}$), one would expect a lower bound on the percolation radius of $r/2$.
The main difference is however, that the percolation result of \cite{PY09} is valid only in two dimension, whereas our result holds in any dimension $d\ge2$.

A pairwise interaction process is locally stable in the following two cases. Firstly if the interaction function $\varphi(x,y)$ is bounded by one; such a process is called \emph{inhibitory}.
Secondly if it has a \emph{hard core radius} $r$, i.e.\ $\varphi(x,y)=0$ whenever $\norm{x-y}\le r$, and $\varphi(x,y)\to 1$ fast enough as  $\norm{x-y}\to \infty$.
The non-percolation result in \cite{PY09} treats only the hard core case. However, unless for the percolation result, the authors proof is quite different from ours.

\section{Proofs}
\label{sec:proofs}

%

The main idea for the proof of percolation is based on techniques close to the contour method in lattice models. In particular, it is a modification of the arguments in \citep{PY09} and \citep{Aristoff12}. 
Let condition \textbf{(P)} be satisfied with constants $r,\delta>0$, and let $R>r$. Choose a $m\in \NN$ such that $m>\sqrt{d}/(R-r)$. Divide $\R^d$ into cubes of length $1/m$. For this sake define
\begin{equation*}
Q_z=\Big\{x\in \R^d\,:\, \norm{x-z}_{\max} \le \frac{1}{2m}\Big\},
\end{equation*}
where we use the maximum norm $\norm{x}_{\max}=\max_{i=1,\dots,d}\abs{x_i}$. Then $\{Q_z,\,z\in \frac{1}{m}\Z^d\}$ covers the whole space $\R^d$. Two cubes $Q_z$ and $Q_{z'}$ are called \emph{neighbours} if $\norm{z-z'}_{\max}=1/m$. A set $S\subset \frac{1}{m}\Z^d$ is called \emph{connected} if for each pair $\{z,z'\}\subset S$ there exists a sequence of neighbouring cubes $Q_{z_1},Q_{z_2},\dots,Q_{z_n}$ with $z=z_1$ and $z'=z_n$. The cardinality of $S$ is denoted by $\abs{S}$. 

\begin{lemma}
\label{lemma:combinatorial}
There exists a constant $c>0$ such that for all $S\subset \frac{1}{m}\Z^d$ with $\abs{S}<\infty$, there exists a $S'\subset S$ with $\abs{S'}\ge c\abs{S}$ and for each pair $z,z'\in S'$ we have $\mathrm{dist}(Q_z,Q_{z'}) \ge r$.
\end{lemma}
\begin{proof}
We can assume $\abs{S}\ge 1$ and then obviously there exists a $S'$ with cardinality one. To identify the constant $c$, consider the following inductive procedure. Choose an arbitrary point $z_1\in S$ and draw a ball around $z_1$ of radius $r+\sqrt{d}/m$. Exclude all points of $S$ which are contained in this ball and denote by $S_1$ the remaining points of $S$. Note that $\mathrm{dist}(Q_{z_1},Q_z)\ge r$ for all $z\in S_1$. Continue this procedure and define $S'=\{z_1,\dots,z_n\}$, the set of the chosen points. At each step we exclude at most $\alpha_d(r+3\sqrt{d}/(2m))^dm^d$ points. Thus
\begin{equation*}
n\ge \max\left(\left\lfloor\frac{\abs{S}}{\alpha_d(r+3\sqrt{d}/(2m))^dm^d}\right\rfloor,1\right),
\end{equation*}
where $\lfloor x\rfloor$ denotes the greatest integer less than or  equal to $x$. Choosing $c=1/(2\alpha_d(r+3\sqrt{d}/(2m))^dm^d)$ yields the claim.
\end{proof}

The following Lemma is the key ingredient for the proof of Theorem~\ref{thm:perc}.

\begin{lemma}
\label{lemma:key}
Assume $\Xi\sim \mu\in \mathcal{G}(\beta,\tilde{\lambda})$ satisfies condition {\bf(P)} with constants $r$ and $\delta$. Let $S\subset\frac{1}{m}\Z^d$ with $\abs{S}<\infty$.
Then if $\beta\ge m^d/\delta$ we have,
\begin{equation}
\label{eq:lemma_key}
\P\Big(\mathrm{dist}\big(\Xi,{\textstyle \bigcup_{z\in S}Q_z}\big)\ge r\Big)\le \left(\frac{\beta\delta}{m^d}\right)^{-c\abs{S}},
\end{equation}
where $c$ is the same constant as in Lemma~\ref{lemma:combinatorial}.
\end{lemma}
\begin{proof}
By Lemma~\ref{lemma:combinatorial} choose a $S'\subset S$ with $\abs{S'}\ge c\abs{S}$ and such that for each pair $z,z'\in S'$ we have $\mathrm{dist}(Q_z,Q_{z'}) \ge r$. Set $n=\abs{S'}$ and denote $S'=\{z_1,\dots,z_n\}$. 
Consider the events
\begin{align*}
A_{S,r}&=\big\{\xi\in\mfn\,:\, \mathrm{dist}\big(\xi,{\textstyle \bigcup_{z\in S}Q_z}\big) \ge r\big\} \quad \text{and} \\
B_{S,r}&=\big\{\xi\cup\{x_1,\dots,x_n\}\in \mfn \colon \xi\in A_{S,r},\, x_i\in Q_{z_i},\, \text{for } i=1,\dots,n\big\}.
\end{align*}
Choose a bounded $\Lambda\subset \R^d$, such that the distance between $S$ and the boundary of $\Lambda$ is at least $r+\sqrt{d}/m$ and fix a boundary condition $\omega\in \mfn_{\Lambda^{\mathsf{c}}}$. Then for $\mu$-a.e.\ $\xi\cup\{x_1,\dots,x_n\}\in B_{S,r}$ we have by \eqref{eq:ci-def2} and condition \textbf{(P)} that
\begin{align}
u_{\Lambda,\omega}\big(&\xi|_\Lambda\cup\{x_1,\dots,x_n\}\big)=u_{\Lambda,\omega}\big(\xi|_\Lambda\big)\lambda\big(x_1\mid\xi|_\Lambda\cup\omega\big) \times \nonumber \\ 
\lambda\big(&x_2\mid \xi|_\Lambda\cup\{x_1\}\cup\omega\big)\cdots\lambda\big(x_n\mid \xi|_\Lambda \cup\{x_1,\dots,x_{n-1}\}\cup\omega\big) \nonumber \\
\ge (&\beta\delta)^nu_{\Lambda,\omega}(\xi|_\Lambda). \label{eq:density}
\end{align}
The partition function can then be bounded from below as
\begin{align*}
c_{\Lambda,\omega} \ge& \sum_{k=0}^\infty\frac{1}{k!}\int_\Lambda\cdots \int_\Lambda \one\big\{\{x_1,\dots,x_k\}\in B_{S,r}\big\}\times \\
& \qquad u_{\Lambda,\omega}\big(\{x_1,\dots,x_k\}\big)\;dx_1\cdots dx_k\\
=& \sum_{k=n}^\infty\frac{1}{k!}\frac{k!}{(k-n)!}\int_\Lambda \cdots \int_\Lambda \one\big\{x_1\in Q_{z_1},\dots,x_n\in Q_{z_n}\big\}\times \\
&\qquad \one\big\{\{x_{n+1},\dots,x_{k}\}\in A_{S,r}\big\}
 u_{\Lambda,\omega}\big(\{x_1,\dots,x_{k}\}\big)\;dx_1\cdots dx_k\\
\ge& \Big(\frac{\beta\delta}{m^d}\Big)^n\sum_{j=0}^\infty \frac{1}{j!}\int_\Lambda\cdots\int_\Lambda \one\big\{\{x_1,\dots,x_j\}\in A_S,r\big\}\times \\
&\qquad u_{\Lambda,\omega}\big(\{x_1,\dots,x_j\}\big)\; dx_1\cdots dx_j,
\end{align*}
where $k!/(k-n)!$ is the number of possibilities to choose a set of $k-n$ elements and $n$ sets of one elements out of a set of $k$ elements, and the last inequality follows by \eqref{eq:density}, by integrating over the cubes $\{Q_{z_i},i=1,\dots,n\}$ and by the change of variable $j=k-n$. Thus, by \eqref{eq:loc-spez}
\begin{equation}
\label{eq:ineq-loc}
\mu_{\Lambda,\omega}(A_{S,r})\le \Big(\frac{\beta\delta}{m^d}\Big)^{-n}.
\end{equation}
Since \eqref{eq:ineq-loc} does not depend on $\omega$, Equation \eqref{eq:dlr} yields
\begin{equation*}
\P\Big(\mathrm{dist}\big(\Xi,{\textstyle \bigcup_{z\in S}Q_z}\big)\ge r\Big)=\P(\Xi\in A_{S,r})\le \Big(\frac{\beta\delta}{m^d}\Big)^{-n} \le \Big(\frac{\beta\delta}{m^d}\Big)^{-c\abs{S}},
\end{equation*}
where the last step follows by $n\ge c\abs{S}$.
\end{proof}

\begin{proof}[Remainder of the proof of Theorem~\ref{thm:perc}.]
We call a set $L\subset \frac{1}{m}\Z^d$ a \emph{loop} if $\abs{L}<\infty$, $L$ is connected and each cube $\{Q_z,\, z\in L\}$ has exactly two neighbours. We say $L\subset\frac{1}{m}\Z^d$ is a \emph{loop around the origin} if the origin is contained in the convex hull of $L$. Denote by $\mathcal{L}_0$  the set of all loops around the origin. 

Consider the loop $L=\{z_1,\dots,z_k\}\in \mathcal{L}_0$ such that $Q_{z_1}$ and $Q_{z_k}$ are neighbours and for $i=1,\dots,k-1$ the cubes $Q_{z_i}$ and $Q_{z_{i+1}}$ are neighbours. Since the origin is contained in the convex hull of $L$, the point $z_1$ lies necessarily in the big cube $[-k/m,k/m]^d$, which gives $(2k+1)^d$ possibilities. For $i=1,\dots,k-1$, the cube $Q_{z_{i+1}}$ is a neighbour of $Q_{z_{i}}$, thus there are at most $2^d$ possibilities to place the cube $Q_{z_{i+1}}$. We conclude, that there are at most $(2k+1)^d2^{d(k-1)}$ loops in $\mathcal{L}_0$ with length $k$.
Consider the event 
\begin{equation*}
A_k=\big\{ \xi\in \mfn\colon \text{There exists a }  L\in \mathcal{L}_0 \text{ with }  \abs{L}=k  \text{ and }  \mathrm{dist}\big(\xi,{\textstyle \bigcup_{z\in L}Q_z}\big)\ge r\big\}.
\end{equation*}
Then for $\beta \ge m^d/\delta$ Lemma~\ref{lemma:key} yields
\begin{align*}
\P(\Xi\in A_k)&=\P\Big(\bigcup_{L\in \mathcal{L}_0,\, \abs{L}=k}\big\{\mathrm{dist}\big(\Xi,{\textstyle \bigcup_{z\in L}Q_z}\big)\ge r\big\}\Big)\\
&\le \sum_{L\in \mathcal{L}_0,\, \abs{L}=k}\P\Big(\mathrm{dist}\big(\Xi,{\textstyle \bigcup_{z\in L}Q_z}\big)\ge r\Big)\le (2k+1)^d2^{d(k-1)}\Big(\frac{\beta\delta}{m^d}\Big)^{-ck}.
\end{align*}
Set $\beta_+=(2^{1/c}m)^d/\delta$ and choose a $\beta>\beta_+$. Then,
\begin{equation}
\label{eq:Borell-Cantelli}
\sum_{k=1}^\infty\P(\Xi \in A_k)\le \sum_{k=1}^\infty (2k+1)^d2^{d(k-1)}\Big(\frac{\beta\delta}{m^d}\Big)^{-ck} <\infty.           
\end{equation}
The first Borel-Cantelli Lemma together with \eqref{eq:Borell-Cantelli} then yield that there are a.s.\ only finitely many loops $L\in \mathcal{L}_0$ with $\mathrm{dist}\big(\Xi,{\textstyle \bigcup_{z\in L}Q_z}\big)\ge r$.

\begin{figure}[h]
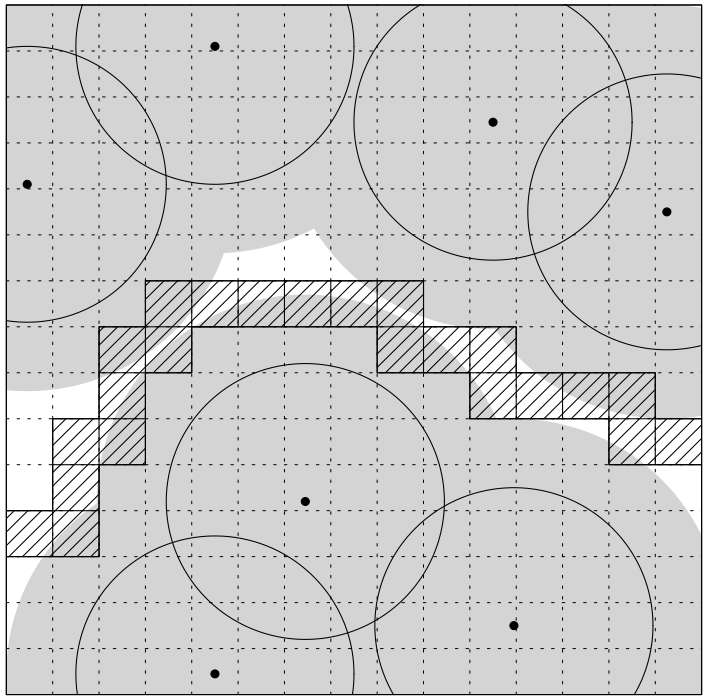

\begin{center}

\caption{Here we choose $d=2$, $r=0.2$, $R=0.3$ and $m=15$. The grey area corresponds to $Z_R(\Xi)$ and the chain of shaded boxes shows one possibility to separate the two components of $Z_R(\Xi)$. Note that the boxes do not intersect with the circles of radius $r$.}
\label{fig:loop}
\end{center}
\end{figure}

The following trick allows us to reduce the problem to the two-dimensional case. 
Consider the set $Z_R^{(2)}(\Xi)=Z_R(\Xi)\cap\R^2$, where we identify $\R^2$ as the plane spanned by the first two canonical basis vectors in $\R^d$. A loop is said to be a \emph{$\R^2$-loop} if the centres of its cubes lie in $\R^2$.

The parameter $m$ is chosen such that if two points $x$ and $x'$ lie in the same connected component of $\R^2\setminus Z_R^{(2)}(\Xi)$, then there exists a sequence of neighbouring cubes with centres $S=\{z_1,\dots,z_n\}\subset \frac{1}{m}\Z^d\cap\R^2$, with $x\in Q_{z_1}$, $x'\in Q_{z_n}$ and $\mathrm{dist}\big(\Xi,{\textstyle \bigcup_{z\in S}Q_z}\big)\ge r$. To see that this is possible consider a path in $\R^2\setminus Z_R^{(2)}(\Xi)$ which connects $x$ with $x'$ and choose all cubes which intersect with this path. Obviously the cubes are connected and the length of the diagonal of a cube is less than $R-r$. Thus, by the reverse triangle inequality, the selected cubes and $\Xi$ are separated by a distance at least $r$.
Figure~\ref{fig:loop} shows a graphical illustration of this procedure.

Assume that $Z_R(\Xi)$ does not percolate, i.e.\ there exists a.s.\ only finite clusters. Obviously the clusters in $Z_R^{(2)}(\Xi)$ are also finite, and each cluster can be surrounded by a $\R^2$-loop $L\in \mathcal{L}_0$ with $\mathrm{dist}\big(\Xi,{\textstyle \bigcup_{z\in L}Q_z}\big)\ge r$. For large $\beta$ this eventually leads to a contradiction.

Note that the choice of a basis in $\R^d$ is by no means important. Thus, the statement of Remark~\ref{rem:subspace}.
\end{proof}

\begin{proof}[Proof of Theorem~\ref{thm:non-perc}]
Let $\Xi$ be a locally stable Gibbs process with conditional intensity $\beta\tilde{\lambda} \le c^*$. Denote $\Lambda_n=[-n,n]^d$, for $n\in \NN$. Note that conditioned on $\Xi\vert_{ \Lambda_n^{\mathsf{c}}}=\omega$, the law of $\Xi\vert_{\Lambda_n}$ is uniquely characterized by the conditional intensity $\beta\tilde{\lambda}(\cdot\mid \cdot\cup\omega)$.

It is a known fact that every locally stable Gibbs process on a bounded domain can be obtained as a dependent random thinning of a Poisson process; see \citep[Remark 3.4]{km00}. In particular, there exists a Poisson process $\Pi_n$ with intensity $c^*$ and a Gibbs process $\tilde{\Xi}_n\sim\mathcal{G}\big(\beta\tilde{\lambda}(\cdot\mid \cdot\cup\omega)\big)$ on $\Lambda_n$, such that $\tilde{\Xi}_n\subset \Pi_n$ a.s. Let $\Xi_n$ be the point process obtained  replacing $\Xi\vert_{\Lambda_n}$ with $\tilde{\Xi}_n$, i.e.\ $\Xi_n=\Xi|_{\Lambda_n^\mathsf{c}}\cup \tilde{\Xi}_n$. Since $\Xi\vert_{\Lambda_n}$ and $\tilde{\Xi}_n$ have the same conditional distribution $\mu_{\Lambda_n,\omega}$, Equation \eqref{eq:dlr} yields that $\Xi$ and $\Xi_n$ have the same distribution for all $n\in \NN$. A standard result (\cite[Theorem~11.1.VII]{dvj08}) then yields that as $n\to \infty$ the sequence $\Xi_n$ has a weak limit $\Xi'$ with the same distribution as $\Xi$. By the same result $\Pi_n\to \Pi$ as $n\to \infty$, where $\Pi$ is a Poisson process on $\R^d$ with intensity $c^*$. Furthermore, since $\tilde{\Xi}_n\subset \Pi_n$ a.s.\ for all $n\in \NN$, we conclude $\Xi'\subset \Pi$ a.s.

Obviously $Z_R(\Xi')\subset Z_R(\Pi)$. Thus if $Z_R(\Pi)$ does not percolate, neither does $Z_R(\Xi')$, and Proposition~\ref{prop:poisson} finishes the proof.
\end{proof}

\section*{Acknowledgements}

The Author thanks Ilya Molchanov and Dominic Schuhmacher for stimulating discussions about the topic. This work was supported by SNF Grant 200021-137527 and DFG-SNF Research Group FOR 916.

\bibliographystyle{plain}

\end{document}